\documentclass[12pt]{amsart}
\usepackage[utf8]{inputenc}
\usepackage[lite, alphabetic, nobysame]{amsrefs}
\usepackage{amsmath,mathtools}
\usepackage{amssymb}
\usepackage{srcltx}
\usepackage{xcolor}
\usepackage{tikz}
\usepackage{xspace}
\usepackage{enumerate}
\usepackage{geometry}
\usepackage{marginnote}

\usepackage{amsmath, amsthm}
\usepackage{amssymb}
\usepackage{subfiles}

\usepackage{todonotes}

\usepackage{cite}

\usepackage{tikz}
\usetikzlibrary{positioning}

\theoremstyle{definition}

\newtheorem{thm}{Theorem}[section]
\newtheorem*{thm*}{Theorem}
\newtheorem{lemma}[thm]{Lemma}
\newtheorem{lem}[thm]{Lemma}
\newtheorem{defn}[thm]{Definition}
\newtheorem{claim}[thm]{Claim}
\newtheorem{prop}[thm]{Proposition}

\newtheorem{remark}[thm]{Remark}
\newtheorem{fact}[thm]{Fact}
\newtheorem{question}[thm]{Question}
\newtheorem{ex}[thm]{Example}

\renewcommand{\subset}{\subseteq}

\renewcommand{\P}{\mathbb{P}}

\newcommand\force{\Vdash}
\newcommand\R{\mathbb{R}}

\newcommand\Q{\mathbb{Q}}
\newcommand\Z{\mathbb{Z}}

\DeclareMathOperator{\dom}{dom}

\newcommand{\set}[2]{ \left\{ #1 :\, #2 \right\} }
\newcommand{\seqq}[2]{ \left\langle #1 :\, #2\right\rangle }

\title[Actions of Polish wreath products]{Strong ergodicity for the $\Gamma$-jump operator and for actions of Polish wreath products}
\author{Assaf Shani}
\date{\today}

\begin{document}

\maketitle
\begin{abstract}
Let $\Gamma$ and $\Delta$ be sufficiently distinct countable groups. We show that there is an orbit equivalence relation $E$, induced by an action of the Polish wreath product group $\Gamma\wr\Gamma$, so that $E$ is generically $F$-ergodic for any orbit equivalence relation $F$ induced by an action of $\Delta\wr\Delta$.
More generally, we establish generic ergodicity between $\Gamma$-jumps and the iterated $\Delta$-jumps, answering a question of Clemens and Coskey~\cite{Clemens-Coskey-jumps}.
The proofs follow a translation between Borel homomorphisms and definable pins, extending ideas in \cite{shani2019-strong-ergodicity} and \cite{Larson_Zapletal_2020}.
\end{abstract}

\section{Introduction}

Let $a\colon G\curvearrowright X$ be an action of a group $G$ on a space $X$. Define the \textbf{orbit equivalence relation} $E_a$, induced by the action $a$, by $x\mathrel{E_a}y\iff \exists g\in G (g\cdot x=y)$.
A common question is: which properties of $G$ are reflected in $E_a$?
Specifically, can we read off of $E_a$ any group-theoretic properties of $G$?
Put in another way, given sufficiently distinct groups $G$ and $H$, can we find an action of $G$ whose induced orbit equivalence relation is sufficiently distinct from any orbit equivalence relation induced by an action of $H$?

In the context of invariant descriptive set theory, the most common scenario is when $G$ is a Polish group, $a\colon G\curvearrowright X$ is a Borel action on a Polish space $X$, and the orbit equivalence relation $E_a$ is studied up to Borel reducibility.
By a result of Becker and Kechris~\cite{Becker_Kechris_1996}, the answers to the questions above do not change if we restrict our attention to continuous actions only. Specifically, for a Polish group $G$ and a Borel action $a\colon G\curvearrowright X$, Becker and Kechris proved that there is some Polish topology on $X$, inducing the same Borel structure, such that the action is continuous with respect to this new topology (see also \cite[4.4]{Gao09}).
We will say that an equivalence relation $E$ is induced by $G$ if it is of the form $E_a$ for some Borel (equivalently, continuous) action of $G$.

Following some recent results and questions in \cite{Clemens-Coskey-jumps}, the focus of this paper is on (full support) Polish wreath product groups of the form $\Gamma\wr\Gamma$, where $\Gamma$ is a countable group.
Let $\Lambda$ and $\Gamma$ be groups. Consider the full support product group $\Lambda^\Gamma$ and the natural shift action $\Gamma\curvearrowright\Lambda^\Gamma$. The \textbf{wreath product} $\Lambda\wr\Gamma$ is defined as the semidirect product $\Gamma\rtimes\Lambda^\Gamma$ with respect to the shift action.
If $\Lambda$ is a Polish group and $\Gamma$ is countable, then $\Lambda\wr\Gamma$ is a Polish group as well. In particular, $\Gamma\wr\Gamma$ is a Polish group for any countable discrete group $\Gamma$.

The situation is interesting already when $\Gamma$ is a ``well-understood'' countable abelian group, such as the group of integers $\Z$; the direct sum $\bigoplus_{n\in\omega}\Z_p=\Z_p^{<\omega}$ of countably many copies of $\Z_p$, where $\Z_p$ is the cyclic group of size $p$; the direct sum $\bigoplus_{p\textrm{ prime}}\Z_p$; or the quasi-cyclic $p$-group $\Z(p^\infty)$, for a prime number $p$ (this group is isomorphic to $\set{z\in\mathbb{C}}{\exists n(z^{p^n}=1)}$).

\begin{ex}\label{example: positive reductions}
For some groups $G$, essentially none of their group theoretic properties are reflected in their induced orbit equivalence relations. Here are a few important examples.
\begin{enumerate}
    \item (Kechris~\cite{Kechris-ctbl-sections-1992}) If $G$ is a locally compact group, then any orbit equivalence relation induced by a continuous action of $G$ is Borel reducible to an orbit equivalence relation induced by a countable discrete group.\footnote{It is open whether this can be true for a group which is not locally compact. See \cite{Kechris-Malicki-Panagiotopoulos-Zielinski}.}
    \item (Gao-Jackson~\cite{Gao-Jackson-abelian-2015}) Suppose $\Gamma$ is a countable discrete group. If $\Gamma$ is abelian then any orbit equivalence relation induced by $\Gamma$ is Borel reducible to $E_0$. \footnote{This is true for a larger class than the abelian groups. Finding the precise class of countable groups for which this holds is a major open problem. See \cite{conley-et-al-borel-asymptotic}.}
    It follows that for any infinite countable discrete groups $\Gamma,\Delta$, any orbit equivalence relation induced by $\Gamma$ is Borel reducible to one induced by $\Delta$. In particular, all the countable abelian groups mentioned above cannot be distinguished by their induced orbit equivalence relations. 
    \item (Ding-Gao~\cite{Ding_Gao_2017}, extending Solecki~\cite{Solecki_1995}) Fix a prime number $p$, and let $(\Z(p^\infty))^\omega$ be the full-support product of $\omega$-copies of $\Z(p^\infty)$. Then any orbit equivalence relation induced by $(\Z(p^\infty))^\omega$ is Borel reducible to $E_0^\omega$. It follows that any orbit equivalence relation induced by $(\Z(p^\infty))^\omega$ is Borel reducible to an action of $\Gamma^\omega$, for any countable infinite group $\Gamma$. In particular, the orbit equivalence relations induced by $(\Z(p^\infty))^\omega$ cannot distinguish it from groups such as $\Z^\omega$, $(\Z(q^\infty))^\omega$, or $(\Z_q^{<\omega})^\omega$.
\end{enumerate}
\end{ex}

We will focus on the opposite phenomenon, when for ``sufficiently different'' groups $G$ and $H$ there is some orbit equivalence relation induced by $G$ which is not Borel reducible to any induced by $H$.
A strong failure of Borel reducibility is given by generic-ergodicity.

Let $E$ and $F$ be equivalence relations on Polish spaces $X$ and $Y$ respectively. A map $f\colon X\to Y$ is a \textbf{homomorphism} from $E$ to $F$ if for any $x,y\in X$, $x\mathrel{E}y\implies f(x)\mathrel{F}f(y)$. If $f$ is a Borel map, we call it a \textbf{Borel homomorphism}, denoted $f\colon E\to_B F$.
Say that $E$ is \textbf{generically} $\mathbf{F}$\textbf{-ergodic} if any Borel homomorphism from $E$ to $F$ maps a comeager subset of $X$ into a single $F$-class. That is, there is a comeager $C\subset X$ so that for any $x,y\in C$, $f(x)\mathrel{F}f(y)$.
Suppose additionally that every $E$-class is meager (this will be the case in all our examples below). Then if $E$ is generically $F$-ergodic, $E$ is not Borel reducible to $F$.

Recall that $f\colon X\to Y$ a \textbf{reduction} of $E$ to $F$ if it is a homomorphism which is injective on the classes, that is, $x\mathrel{E}y\iff f(x)\mathrel{F}f(y)$. We say that $E$ is \textbf{Borel reducible to} $F$ if there is a Borel map which is a reduction of $E$ to $F$.

\begin{ex}\label{example: generic ergodicity}
\begin{enumerate}
    \item Hjorth's turbulence condition for a continuous action $a\colon G\curvearrowright X$ implies that $E_a$ is generically $F$-ergodic for any orbit equivalence relation $F$ induced by an action of a closed subgroup of $S_\infty$ (see \cite{Hjo00}, \cite[12.5]{Kechris_1997-classification-problems}, or \cite[10.4]{Gao09}). Recall that the closed subgroups of $S_\infty$ are characterized, for example, as the Polish groups which admit a countable neighborhood base of the identity comprised of open subgroups (see \cite[2.4.1]{Gao09}). On the other hand, for Polish groups such as the Banach spaces $c_0$ and $l_p$ ($1\leq p<\infty$), their natural action on $\R^\omega$ is turbulent (see \cite[10.5]{Gao09}). 
    \item Allison and Panagiotopoulos~\cite{Allison-Panagiotopoulos-2021} introduced the unbalancedness condition for a continuous action $a\colon G\curvearrowright X$, and proved that it implies that $E_a$ is generically $F$-ergodic for any orbit equivalence relation $F$ induced by a TSI group. One corollary of this result is the following. Let $\Gamma$ be a countable infinite discrete group. There is an orbit equivalence relation $E$ induced by an action of $\Gamma\wr\Gamma$ so that $E$ is generically $F$-ergodic for any orbit equivalence relation $F$ induced by an action of the product group $\Gamma^\omega$. This draws a sharp distinction between wreath product and products, in terms of their possible induced orbit equivalence relations. 
\end{enumerate}
\end{ex}
We show that for sufficiently distinct countable groups $\Gamma$ and $\Delta$ there is an orbit equivalence relation induced by $\Gamma\wr\Gamma$ which is strongly $F$-ergodic for any orbit equivalence relation $F$ induced by $\Delta\wr\Delta$.

\begin{thm}\label{thm: wreath prod strong ergodicity}
Suppose $\Gamma$ and $\Delta$ are countable groups such that for any group homomorphism $\phi$ from $\Gamma$ to a quotient of a subgroup of $\Delta$,
\begin{itemize}
    \item the Image of $\phi$ is finite;
    \item the Kernel of $\phi$ is isomorphic to $\Gamma$.
\end{itemize}
Then there is an orbit equivalence relation induced by $\Gamma\wr\Gamma$ which is strongly $F$-ergodic for any orbit equivalence relation $F$ induced by $\Delta\wr\Delta$.
\end{thm}
\begin{ex}\label{example: main examples} A few interesting examples where the conclusion of Theorem~\ref{thm: wreath prod strong ergodicity} holds are as follows.
\begin{enumerate}
    \item If all group homomorphisms from $\Gamma$ to a quotient of a subgroup of $\Delta$ are trivial.
    \item $\Gamma=\Z(p^\infty)$ and $\Delta=\Z(q^\infty)$, for distinct primes $p,q$. (In contrast to the situation for product groups $\Gamma^\omega$, $\Delta^\omega$, mentioned in Example~\ref{example: positive reductions}~(3))
    \item Let $\Gamma$ and $\Delta$ be two distinct groups among
\begin{equation*} \mathbb{Z}, \bigoplus_{p\textrm{ prime}}\mathbb{Z}_p,\textrm{ or } \mathbb{Z}_p^{<\omega}\textrm{ for some prime }p.    \end{equation*}
The only instances which do not satisfy the conditions of Theorem~\ref{thm: wreath prod strong ergodicity} are when $\Gamma=\bigoplus_{p\textrm{ prime}}\mathbb{Z}_p$. These will be dealt with in Section~\ref{subsec: example}.
\end{enumerate}
\end{ex}

The particular action of $\Gamma\wr\Gamma$ in Theorem~\ref{thm: wreath prod strong ergodicity} comes from the $\Gamma$-jump operation defined by Clemens and Coskey~\cite{Clemens-Coskey-jumps}. This is also the action used in Example~\ref{example: generic ergodicity}~(2) mentioned above.
Let $E$ be an equivalence relation on $X$, and $\Gamma$ a countable group. Clemens and Coskey defined the \textbf{$\Gamma$-jump of $E$}, denoted $E^{[\Gamma]}$, as the equivalence relation on $X^\Gamma$ defined by
\begin{equation*}
        x \mathrel{E^{[\Gamma]}} y \iff (\exists \gamma\in\Gamma) (\forall \alpha\in\Gamma) x(\gamma^{-1}\alpha) \mathrel{E} y(\alpha).
\end{equation*}
For a set $\Gamma$, the \textbf{product equivalence relation} $E^\Gamma$ is defined on $X^\Gamma$ by $x\mathrel{E^\Gamma}y\iff(\forall\gamma\in\Gamma){x(\gamma)\mathrel{E}y(\gamma)}$. 
For a countable group $\Gamma$, $\Gamma$ acts on $X^\Gamma$ via the shift action $\gamma\cdot\seqq{x_\alpha}{\alpha\in\Gamma}= \seqq{x_{\gamma^{-1}\alpha}}{\alpha\in\Gamma}$. Then
\begin{equation*}
        x \mathrel{E^{[\Gamma]}} y \iff (\exists \gamma\in\Gamma) \gamma\cdot x \mathrel{E^\Gamma} y.
\end{equation*}
The iterated $\Gamma$-jumps are defined in \cite{Clemens-Coskey-jumps} recursively as follows.
\begin{itemize}
    \item $J_0^{[\Gamma]}(E)=E$;
    \item $J_{\alpha+1}^{[\Gamma]}(E)=(J_\alpha^{[\Gamma]}(E))^{[\Gamma]}$;
    \item $J_\alpha^{[\Gamma]}(E)=\left(\bigoplus_{\beta<\alpha}J_{\beta}^{[\Gamma]}(E)\right)^{[\Gamma]}$\, for limit $\alpha$.
\end{itemize}
Finally, $J_\alpha^{[\Gamma]}$ is defined as $J_\alpha^{[\Gamma]}(\Delta(2))$, where $\Delta(2)$ is the equality relation on $\{0,1\}$.

\begin{ex}
\begin{enumerate}
    \item $J_1^{[\Gamma]}$ is the orbit equivalence relation induced by the standard shift action of $\Gamma$ on $2^\Gamma$.
    \item If $E$ is the orbit equivalence relation induced by an action of $\Lambda$, then $E^{[\Gamma]}$ is the orbit equivalence relation of a natural action of $\Lambda\wr\Gamma$. (See \cite{Clemens-Coskey-jumps}.)
    \item In particular, $J_2^{[\Gamma]}$ is an orbit equivalence relation induced by an action of $\Gamma\wr\Gamma$. This is the orbit equivalence relation witnessing Theorem~\ref{thm: wreath prod strong ergodicity}.
\end{enumerate}
\end{ex}

Clemens and Coksey~\cite{Clemens-Coskey-jumps} established a close relationship between actions of wreath products and the $\Gamma$-jump operations.
For example, a consequence of their results is the following. Suppose $E$ is a Borel equivalence relation induced by an action of $\Gamma\wr\Gamma$, for a countable group $\Gamma$.
Then $E$ is Borel reducible to $J^{[\Gamma]}_\alpha$ for some countable ordinal $\alpha$.
They also showed that the iterated $\Z$-jump $J_\alpha^{[\Z]}$ is Borel bireducible with the isomorphism relation on countable scattered linear orders of rank $1+\alpha$.

Clemens and Coskey investigated for which groups $\Gamma$ the operation $E\mapsto E^{[\Gamma]}$ is indeed a jump operator on Borel equivalence relations (see the definitions in \cite[Section 1]{Clemens-Coskey-jumps}). For example, they show that it is a jump operator for $\Z$ and $\Z_p^{<\omega}$ and it is not a jump operator for $\Z(p^\infty)$. For $\bigoplus_{p\textrm{ prime}}\Z_p$ it remains open.

One problem left open in \cite{Clemens-Coskey-jumps} is whether there are countable groups $\Gamma,\Delta$ so that some $\Gamma$-jump $J_\alpha^{[\Gamma]}$ is not Borel reducible to \emph{any} $\Delta$-jump $J_\beta^{[\Delta]}$, and vice-versa (see \cite[Question 2]{Clemens-Coskey-jumps}). 
We show that this is the case for $\Gamma,\Delta$ as in Example~\ref{example: main examples} parts (2) and (3).
Moreover, the criterion on $\Gamma$ and $\Delta$ from Theorem~\ref{thm: wreath prod strong ergodicity} (in terms of group homomorphisms) implies that $J_2^{[\Gamma]}$ is generically $J_\beta^{[\Delta]}$-ergodic, and therefore not Borel reducible to $J_\beta^{[\Delta]}$, for any countable ordinal $\beta$. This is in fact the main result behind Theorem~\ref{thm: wreath prod strong ergodicity}.

\begin{thm}\label{thm: main iterated jump ergodicity}
Suppose $\Gamma$ and $\Delta$ are as in Theorem~\ref{thm: wreath prod strong ergodicity}.
Let $E$ be a generically ergodic countable Borel equivalence relation. Then for any countable ordinal $\beta$, $E^{[\Gamma]}$ is generically $J_\beta^{[\Delta]}$-ergodic. 
In particular, $J_2^{[\Gamma]}$ is generically $J_\beta^{[\Delta]}$-ergodic for any countable ordinal $\beta$.
\end{thm}

The main technical result, from which the above theorems will be deduced, is the following, lifting $F$-ergodicity to $F^{[\Delta]}$-ergodicity. In a sense, it shows that the $\Gamma$-jump and $\Delta$-jump operators are ``perpendicular'', in terms of Borel homomorphisms, for sufficiently different $\Gamma$ and $\Delta$, in terms of group homomorphisms.

\begin{thm}\label{thm: main technical gen-erg thm}
Let $\Gamma$ and $\Delta$ be countable infinite groups and $F$ an analytic equivalence relation. 
Assume that for any group homomorphism $\phi$ from $\Gamma$ to a quotient of a subgroup of $\Delta$,
\begin{itemize}
    \item the image of $\phi$ is finite;
    \item if $K$ is the Kernel of $\phi$, then $E^{[K]}$ is generically $F$-ergodic, for any generically ergodic countable Borel equivalence relation $E$.
\end{itemize}
Then $E^{[\Gamma]}$ is generically $F^{[\Delta]}$-ergodic, for any generically ergodic countable Borel equivalence relation $E$.
\end{thm}


\begin{question}
\begin{enumerate}
    \item Characterize the countable infinite groups $\Gamma,\Delta$ for which there is an orbit equivalence relation induced by $\Gamma\wr\Gamma$ which is generically $F$-ergodic with respect to any orbit equivalence $F$ induced by $\Delta\wr\Delta$.
    \item Characterize the countable infinite groups $\Gamma,\Delta$ for which there is some countable ordinal $\alpha$ so that $J_\alpha^{[\Gamma]}$ is not Borel redicuble to any $J_\beta^{[\Delta]}$.
\end{enumerate}
\end{question}
\begin{remark}
\cite[Theorem 1.3]{shani2019-strong-ergodicity} characterizes the countable infinite groups $\Gamma,\Delta$ for which $J_2^{[\Gamma]}$ is generically $J_2^{[\Delta]}$-ergodic.
\end{remark}

\begin{question}
Let $\Gamma$ be a countable infinite group. Is $J_3^{[\Gamma]}$ Borel reducible to an orbit equivalence relation induced by $\Gamma\wr\Gamma$?
\end{question}

The proofs below rely on an analysis of definable sets in a particular symmetric model, where the axiom of choice fails. (Even very weak fragments of choice, such as the axiom of dependent choices for reals, fail in this model). The following lemma provides the translation.

\begin{lem}\label{lemma: symmetric-model-generic-ergod-introduction}
Let $E$ be a generically ergodic analytic equivalence relation on a Polish space $X$, which is classifiable by countable structures. Let $F$ be an analytic equivalence relation on a Polish space $Y$. There is a symmetric model of the form $V(A)$ (for some set $A$ in a Cohen-real extension of a ground model $V$), and a poset $\P$ in $V(A)$, so that the following are equivalent.
\begin{enumerate}
    \item $E$ is generically $F$-ergodic;
    \item suppose $\sigma\in V(A)$ is a $\P$-name so that $(\P,\sigma)$ is an $F$-pin and $\sigma$ is definable in $V(A)$ from $A$ over $V$. Then there is $y_0\in Y\cap V$ so that $\P\force \sigma\mathrel{F} \check{y}_0$.
\end{enumerate}
\end{lem}
See Section~\ref{section: prelims - classifying invariants} for any undefined terms. The lemma is proven below (Lemma~\ref{lem;symm-model-ergod}) in greater generality, in the context of proper ideals, following \cite{Zapletal-2008-forcing-idealized}. 
The lemma extends \cite[Lemma~2.5]{shani2019-strong-ergodicity}, in which the equivalence relation $F$ was assumed to be classifiable by countable structures as well. The key point is using pins, following \cite{Larson_Zapletal_2020}, instead of only classifying invariants. This allows us to deal with homomorphisms to arbitrary analytic equivalence relations in Theorem~\ref{thm: main technical gen-erg thm}.


\section{Proofs of Theorem~\ref{thm: wreath prod strong ergodicity}, Theorem~\ref{thm: main iterated jump ergodicity}, and Example~\ref{example: main examples}}

Before proceeding to the proof of Theorem~\ref{thm: main technical gen-erg thm}, we explain here how the rest of the results mentioned in the introduction follow from it, and from the results in \cite{Clemens-Coskey-jumps}.

\subsection{Theorem~\ref{thm: main iterated jump ergodicity}}\label{subsec: proof of iterated jump erg}

Fix $\Gamma$ and $\Delta$ as in Theorem~\ref{thm: wreath prod strong ergodicity}. Fix a generically ergodic countable Borel equivalence relation $E$. We show by induction on $\alpha<\omega_1$ that $E^{[\Gamma]}$ is generically $J^{[\Delta]}_\alpha$-ergodic.
Note that, since $E$ is generically ergodic, so is $E^{[\Gamma]}$, and so it is generically $J_0^{[\Delta]}$-ergodic.

Assume that $E^{[\Gamma]}$ is generically $J^{[\Delta]}_\alpha$-ergodic. 
For any group homomorphism $\phi$ from $\Gamma$ to a quotient of a subgroup of $\Delta$, by assumption, the image of $\phi$ is finite, and if $K$ is the Kernel of $\phi$, then $K\simeq \Gamma$. In particular, $E^{[K]}$ is generically $J^{[\Delta]}_\alpha$-ergodic. 
It follows from Theorem~\ref{thm: main technical gen-erg thm} that $E^{[\Gamma]}$ is generically $J^{[\Delta]}_{\alpha+1}=(J^{[\Delta]}_\alpha)^{[\Delta]}$-ergodic. 

Assume now that $\alpha$ is a countable limit ordinal and assume that $E^{[\Gamma]}$ is generically $J^{[\Delta]}_\beta$-ergodic for any $\beta<\alpha$. 
Then $E^{[\Gamma]}$ is generically $\bigoplus_{\beta<\alpha}J_{\beta}^{[\Delta]}$-ergodic.
As before, it follows from Theorem~\ref{thm: main technical gen-erg thm} that $E^{[\Gamma]}$ is generically $J_{\alpha}^{\Delta}=\left(\bigoplus_{\beta<\alpha}J_{\beta}^{[\Delta]}\right)^{[\Delta]}$-ergodic.

\subsection{Theorem~\ref{thm: wreath prod strong ergodicity}}\label{subsec: proof of main wreath prod erg} 
Note first that it is sufficient to show generic ergodicity with respect to more complex equivalence relations. Specifically, if $E,F_1,F_2$ are equivalence relations on Polish spaces $X,Y_1,Y_2$ respectively, $E$ is generically $F_2$-ergodic, and $F_1$ is Borel reducible to $F_2$, then $E$ is generically $F_1$-ergodic. 
Note also that if $E$ is generically $F$-ergodic, and $f$ is a partial Borel homomorphism defined on an invariant comeager subdomain of $E$, then $f$ must send a comeager set into a single $F$-class.

Let $\Gamma$ and $\Delta$ be as in Theorem~\ref{thm: wreath prod strong ergodicity}.
Recall the infinite wreath-product $\Delta^{\wr\omega}$ as in \cite{Clemens-Coskey-jumps}. $\Delta^{\wr\omega}$ is isomorphic to $\mathrm{Aut}(T_\Delta)$, the automorphism group of the full $\Delta$-tree (see \cite[Section~4]{Clemens-Coskey-jumps}.) Note that $\Delta\wr\Delta$ is isomorphic to a closed subgroup of $\Delta^{\wr\omega}$. The following proposition implies the conclusion of Theorem~\ref{thm: wreath prod strong ergodicity}.
\begin{prop}
$J_2^{[\Gamma]}$ is generically $F$-ergodic for any orbit equivalence $F$ induced by an action of a closed subgroup of $\Delta^{\wr\omega}$.
\end{prop}
Let $F$ be an orbit equivalence relation induced by an action of a closed subgroup of $\Delta^{\wr\omega}$. If $F$ is Borel, then, by \cite[Theorem~2]{Clemens-Coskey-jumps}, $F$ is Borel reducible to $J_\beta^{[\Delta]}$ for some countable ordinal $\beta$, and the result follows from Theorem~\ref{thm: main iterated jump ergodicity}.

For the general case, we may replace $F$ with a Borel equivalence relation by passing to a comeager subdomain of $J_2^{[\Gamma]}$. We use some standard facts about orbit equivalence relations, as in \cite{Gao09} or \cite{Becker_Kechris_1996}.
Let $f$ be a Borel homomorphism from $J_2^{[\Gamma]}$ to $F$, where $F$ is the orbit equivalence relation induced by some action of $G$ on $Y$, where $G$ is a closed subgroup of $\Delta^{\wr\omega}$.
Let $X$ be the domain of $J_2^{[\Gamma]}$. 
For each $x\in X$, the orbit $[f(x)]_F$ is Borel.
We may find an invariant comeager set $X_0\subset X$ and a countable ordinal $\gamma$ so that the orbit $[f(x)]_F$ is $\mathbf{\Pi}^0_\gamma$ for any $x\in X_0$.
Let $Y_0$ be the set of $y\in Y$ so that $[y]_F$ is $\mathbf{\Pi}^0_\gamma$. Then $Y_0$ is invariant under $F$, $F\restriction Y_0$ is a Borel equivalence relation, and $f(x)\in Y_0$ for any $x\in X_0$.
Now $F\restriction Y_0$ is the orbit equivalence relation induced by a Borel action of $G$, so by \cite[Theorem 2]{Clemens-Coskey-jumps}, $F\restriction Y_0$ is Borel reducible to $J_\beta^{[\Delta]}$ for some countable ordinal $\beta$.
The result now follows from Theorem~\ref{thm: main iterated jump ergodicity}.

\subsection{Example~\ref{example: main examples}}\label{subsec: example}

The only instances of Example~\ref{example: main examples} which do not satisfy the conditions of Theorem~\ref{thm: wreath prod strong ergodicity} are when $\Gamma=\bigoplus_{p\textrm{ prime}}\mathbb{Z}_p$ and either
\begin{enumerate}
    \item $\Delta=\mathbb{Z}_q^{<\omega}$, or
    \item $\Delta=\Z$.
\end{enumerate}
 
Fix a generically ergodic countable Borel equivalence relation $E$. We prove that $E^{[\Gamma]}$ is generically $J_\beta^{[\Delta]}$-ergodic for any countable ordinal $\beta$ (establishing the conclusion of Theorem~\ref{thm: main iterated jump ergodicity}). The conclusion of Theorem~\ref{thm: wreath prod strong ergodicity}, that the orbit equivalence relation $J_2^{[\Gamma]}$ is not Borel reducible to any orbit equivalence relation induced by $\Delta\wr\Delta$, then follows as in Section~\ref{subsec: proof of main wreath prod erg} above. 

Consider first case (1). Note that $\bigoplus_{p\textrm{ prime}}\mathbb{Z}_p=(\bigoplus_{q\neq p}\mathbb{Z}_p)\times \mathbb{Z}_q$, and any group homomorphism from $\bigoplus_{q\neq p}\mathbb{Z}_p$ to a quotient of a subgroup of $\mathbb{Z}_q^{<\omega}$ is trivial.
It therefore suffices to show the following.
\begin{prop}
Suppose $\Gamma=\Gamma_0\times G$ where $G$ is a finite group and all group homomorphisms from $\Gamma_0$ to a quotient of a subgroup of $\Delta$ are trivial. Let $E$ be a generically ergodic countable Borel equivalence relation. Then $E^{[\Gamma]}$ is generically $J_\alpha^{[\Delta]}$-ergodic for all $\alpha<\omega_1$. 
\end{prop}
\begin{proof}
The proof is inductive, as in Section~\ref{subsec: proof of iterated jump erg}. Suppose the proposition holds for $\alpha$.
For any group homomorphism $\phi\colon \Gamma\to\Delta$, its kernel $K$ is equal to $\Gamma_0\times G'$ for some finite group $G'$.
In particular, $K=\Gamma_0\times G'$ and $\Delta$ satisfy the assumptions of the propositions, so $E^{[K]}$ is generically $J_\alpha^{[\Delta]}$-ergodic, by the inductive assumption.
Note also that any group homomorphism from $\Gamma$ to a quotient of a subgroup of $\Delta$ has finite image.
By Theorem~\ref{thm: main technical gen-erg thm} we conclude that $E^{[\Gamma]}$ is generically $J_{\alpha+1}^{[\Delta]}$-ergodic. The limit case is similar. 
\end{proof}

The proof of case (2) is similar. We carry the following stronger inductive hypothesis at stage $\alpha$:
\begin{claim}
For any infinite set of primes $P$, $E^{[\bigoplus_{p\in P}\Z_p]}$ is generically $J_\alpha^{[\Delta]}$-ergodic.
\end{claim}
Assume the claim holds for $\alpha$, and fix some infinite set of primes $P$. Let $\phi$ be a group homomorphism from $\bigoplus_{p\in P}\Z_p$ to a quotient of $\Z$. We show that the assumptions in Theorem~\ref{thm: main technical gen-erg thm} hold. Then the theorem implies that $E^{[\bigoplus_{p\in P}\Z_p]}$ is generically $(J_\alpha^{[\Delta]})^{[\Delta]}=J_{\alpha+1}^{[\Delta]}$-ergodic, as required.
If $\phi$ is trivial, we are done. If $\phi$ is non-trivial, then it must be into a non-trivial quotient of $\Z$, which is finite. Furthermore, its kernel is of the form $K=\bigoplus_{p\in P'}\Z_p$ for an infinite set $P'\subset P$. By the inductive assumption, $E^{[K]}$ is generically $J_\alpha^{[\Delta]}$-ergodic, as required.

\section{Preliminaries}

In this section we cover some background and some lemmas towards the proof of Theorem~\ref{thm: main technical gen-erg thm}.

Familiarity with the basics of forcing, as can be found in \cite{Halbeisen-book-2017}, \cite{Jech2003}, or \cite{Kunen2011}, will be assumed for the remainder of this paper. We denote by $V$ the ground model over which we force. Given a forcing poset $\P$, a $\P$-name $\tau$, and a generic filter $G\subset\P$, $\tau[G]$ denotes the interpretation of $\tau$ according to $G$. Often the generic filter $G$ is identified with some set $x$, (for example, a generic real) in which case we may write $\tau[x]$ instead of $\tau[G]$.

Given two posets $\P,\Q$, and filters $G\subset\P$, $H\subset \Q$, say that they are \textbf{mutually generic} if $G\times H\subset \P\times\Q$ is a generic filter. We will use often and implicitly the mutual genericity lemma, that if $G\subset \P$ is generic over $V$ and $H\subset\Q$ is generic over $V[G]$ then $G$ and $H$ are mutually generic.

\subsection{Classifying invariants}\label{section: prelims - classifying invariants}
Recall that $E$ is {\bf classifiable by countable structures} if $E$ is Borel reducible to an isomorphism relation on countable structures.
Equivalently, $E$ is Borel reducible to an orbit equivalence relation induced by a closed subgroup of $S_\infty$. The reader is referred to \cite[12.3]{Kano08}, \cite{Gao09}, \cite{Kechris_1997-classification-problems}, or \cite{Hjo00}, for background on classification by countable structures. 
For example, the iterated jumps $J_\alpha^{[\Gamma]}$ of Clemens and Coskey are all classifiable by countable structures (see \cite{Clemens-Coskey-jumps}).

Let $E$ be an analytic equivalence relation on a Polish space $X$. 
A {\bf complete classification of $E$} is an assignment $x\mapsto A_x$ such that for any $x,y\in X$,
\begin{equation*}
x\mathrel{E}y \iff A_x=A_y.
\end{equation*}
The sets $A_x$ are called complete invariants for $E$, or \textbf{classifying invariants} for $E$.

The key property of equivalence relations which are classifiable by countable structures which will be used in this paper is that they admit a complete classifications which is absolute in the following way (see \cite[Fact 2.5]{Sha18}).
Say that the map $x\mapsto A_x$ is an \textbf{absolute classification} if it is definable by some set-theoretic formula in an absolute way. That is, for any model $V$ (containing the parameters)
\begin{itemize}
    \item this formula defines a complete classification of $E$, and
    \item for $x$ in $V$, $A_x$ is computed the same in $V$ and any generic extension of $V$. 
\end{itemize}

Any equivalence relation which is classifiable by countable structures admits an absolute classification via the Scott analysis, see \cite[Lemma 2.4]{Friedman2000}, \cite[Chapter 12.1]{Gao09}, or \cite[Section 3.1]{Ulrich-Rast-Laskowski-2017}.
Let us mention a few examples of absolute classifications which will be used below.

\begin{ex}
Let $E$ be a countable Borel equivalence relation on a Polish space $X$.
\begin{itemize}
    \item The trivial classification, $x\mapsto [x]_E$, is an absolute classification. Since each $E$-class $[x]_E$ is countable, it does not change after forcing.
    \item The equivalence relation $E^\omega$ on $X^\omega$ is no longer countable, and the naive classification $x\mapsto [x]_{E^\omega}$ fails to be absolute.  $E^\omega$ is classifiable by countable structures, and admits the absolute complete classification $x\mapsto\seqq{[x(n)]_E}{n<\omega}$. The classifying invariants are countable sequences of countable subsets of $X$.

\end{itemize}
\end{ex}

Next we consider classifying invariants for the $\Gamma$-jumps, which will be crucial in our study of the $\Gamma$-jumps in this paper.
\begin{ex}\label{example: abs class of jumps}
Fix a countable group $\Delta$, and an analytic equivalence relation $F$.
Assume that $F$ is classified by the absolute classification $y\mapsto B_y$. The product equivalence relation $F^\Delta$ admits the absolute complete classification $\seqq{x(\zeta)}{\zeta\in\Delta}\mapsto\seqq{B_{x(\zeta)}}{\zeta\in\Delta}$, where the invariants are $\Delta$-sequences of $F$-classifying invariants.
The group $\Delta$ acts on such $\Delta$-sequences of invariants by $\delta\cdot\seqq{B_\zeta}{\zeta\in\Delta}=\seqq{B_{\delta^{-1}\zeta}}{\zeta\in\Delta}$. 
The $\Delta$-jump $F^{[\Delta]}$ is classified by the absolute complete classification
\begin{equation*}    \seqq{x(\zeta)}{\zeta\in\Delta}\mapsto \Delta\cdot \seqq{B_{x(\zeta)}}{\zeta\in\Delta}=\set{\delta\cdot\seqq{B_{x(\zeta)}}{\zeta\in\Delta}}{\delta\in\Delta}. \end{equation*}

A key example is the following. If $E$ is a countable Borel equivalence relation on $X$, then $E^{[\Delta]}$ on $X^\Delta$ is classified by the absolute map 
\begin{equation*}    \seqq{x(\zeta)}{\zeta\in\Delta}\mapsto \Delta\cdot \seqq{[x(\zeta)]_{E}}{\zeta\in\Delta}. \end{equation*}

\end{ex}

\subsection{Pins for $\Gamma$-jumps}
Recall the definition of pins for equivalence relations. The following definitions and terminology are from {\cite[Section~2.1]{Larson_Zapletal_2020}} (see also \cite{Kano08}). 
\begin{defn}\label{def;pinned}
Let $E$ be an analytic equivalence relation on a Polish space $X$. Let $\mathbb{P}$ be a poset and $\tau$ a $\mathbb{P}$-name forced to be in $X$. 
\begin{itemize}
    \item The name $\tau$ is $E$-pinned if $\mathbb{P}\times \mathbb{P}$ forces that $\tau_{l}$ is $E$-equivalent to $\tau_{r}$, where $\tau_{l}$ and $\tau_{r}$ are the interpretation of $\tau$ using the left and right generics respectively.
    \item If $\tau$ is $E$-pinned the pair $\left<\mathbb{P},\tau\right>$ is called an $\boldsymbol{E}$-\textbf{pin}.
    \item An $E$-pin $\left<\mathbb{P},\tau\right>$ is \textbf{trivial} if there is some $x\in X$ such that $\mathbb{P}\Vdash \tau\mathrel{E}\check{x}$.
    \item Given two $E$-pins $\left<\mathbb{P},\sigma\right>$ and $\left<\mathbb{Q},\tau\right>$, say that they are $E$-equivalent\footnote{The extension of $E$ to pins is denoted in \cite{Larson_Zapletal_2020} as $\bar{E}$. Here we use $E$, due to laziness.}, $\left<\mathbb{P},\sigma\right>\mathrel{E}\left<\mathbb{Q},\tau\right>$, if $\mathbb{P}\times \mathbb{Q} \Vdash \sigma \mathrel{E}\tau$.
\end{itemize}
\end{defn}

An $E$-pin $\left<\P,\tau\right>$ describes an $E$-class in a further generic extension by $\P$. The pin is trivial if this is simply the $E$-class of some element in the ground model. The reader is referred to \cite{Larson_Zapletal_2020} and \cite{Kano08} for more about pins. One useful fact is that, for an $E$-pin $\left<\P,\tau\right>$, for any two generic filter $G,H\subset\P$ (not necessarily mutually generic), $\tau[G]$ and $\tau[H]$ are $E$-related in any model containing both. (See \cite[Proposition 2.1.2]{Larson_Zapletal_2020}.)

Let $E$ be an analytic equivalence relation. Larson and Zapletal proved \cite[Theorem 2.6.2]{Larson_Zapletal_2020} that if $\P$ is a \textbf{reasonable} poset then any $E$-pin $\left<\P,\tau\right>$ is trivial. See \cite[Definition~2.6.1]{Larson_Zapletal_2020} for the definition of reasonable, due to Foreman and Magidor. Any proper poset, and so any c.c.c. poset, is reasonable. For the applications in this paper, we will only need the fact that there are no non-trivial pins $\left<\P,\tau\right>$ when $\P$ is the Cohen-real poset.
See \cite[Section 2.6]{Larson_Zapletal_2020} for a thorough discussion, and open questions, about which posets may carry a non-trivial pin.

\begin{remark}
Suppose $\left<\P,\tau\right>$ is an $E$-pin and $f\colon E\to_B F$ is a Borel homomorphism from $E$ to $F$. Let $\sigma$ be a $\P$-name which is forced to be $f(\tau)$. Then $\left<\P,\sigma\right>$ is an $F$-pin.
More generally, $f$ extends to a map (defined in an absolute way) from $E$-pins to $F$-pins, sending $E$-equivalent pins to $F$-equivalent pins.
\end{remark}

Let $F$ be an analytic equivalence relation on a Polish space $Y$ and $\Delta$ a countable group. 
Larson and Zapletal \cite[Example 2.3.12]{Larson_Zapletal_2020} classified the $F^{[\Delta]}$-pins, in terms of $F$-pins.
We will use the following lemma, which is an immediate consequence of \cite[Theorem 2.3.8]{Larson_Zapletal_2020}.
\begin{lemma}\label{lemma: pins from jump and product}
Suppose $(\Q,\sigma)$ is an $F^{[\Delta]}$-pin. Then there is some $q\in\Q$ such that $(\Q\restriction q,\sigma)$ is an $F^\Delta$-pin.
\end{lemma}
\begin{proof}
By assumption, it is forced by $\Q\times\Q$ that $\sigma_l \mathrel{F^{[\Delta]}}\sigma_r$, so there is $\delta\in\Delta$ such that $\delta\cdot\sigma_l \mathrel{F^\Delta} \sigma_r$.
There is $\delta\in\Delta$ and $(p,q)\in\Q\times\Q$ forcing that $\delta\cdot\sigma_l \mathrel{F^\Delta} \sigma_r$. The condition $q$ satisfies the conclusion of the lemma. Indeed, given two generics $G_1,G_2$ extending $q$, take $G_3$ generic over them, extending $p$. We conclude that $\delta\cdot\sigma[G_3] \mathrel{F^\Delta} \sigma[G_1]$ and $\delta\cdot\sigma[G_3] \mathrel{F^\Delta} \sigma[G_2]$, and therefore $\sigma[G_1] \mathrel{F^\Delta}\sigma[G_2]$.
\end{proof}


\subsection{Symmetric models and Borel homomorphisms}\label{subsec: symm models and homomorphisms}

We will study Borel homomorphism by studying definable pins in symmetric models, according to Lemma~\ref{lem;symm-model-ergod}. On the one hand, this is a generalization of ideas in \cite[Section 2.8]{Larson_Zapletal_2020}, where homomorphisms from $=^+$ (to some analytic equivalence relation) were analyzed by studying pins in the Solovay model.

For different equivalence relations $E$, a particularly chosen symmetric model is constructed, following \cite{Sha18,shani2019-strong-ergodicity}, in which homomorphisms from $E$ correspond to definable pins in this model. This is also reminiscent of the arguments in \cite[Section~6.1.3]{Kanovei-Sabok-Zapletal-2013}, where homomorphisms from $=^+$ were analyzed using (implicitly) the symmetric Cohen model.

The relevant models will be of the form $V(A)$, the minimal transitive ZF extension of $V$ in which $A$ is a member, where $A$ is a set in some generic extension of $V$. Typically, $A$ will be a classifying invariant of some generic real. Sets in $V(A)$ are definable, in $V(A)$, using parameters from $V$ and the transitive closure of $A$.
\begin{fact}
For any set $X\in V(A)$, there is some formula $\psi$, a parameter $v\in V$ and finitely many parameters $\bar{a}$ from the transitive closure of $A$, so that $X$ is defined, in $V(A)$, as the unique set satisfying $\psi(X,A,\bar{a},v)$.
In this case, say that $X$ is \textbf{definable from $A$ and $\bar{a}$ over $V$}.
\end{fact}
(See \cite[Section~2]{shani2019-strong-ergodicity}.) The possible defining parameters of some set in $V(A)$ will often be the key question. Of particular interest are the sets definable without parameters, that is, definable from $A$ alone over $V$.

Let $X$ be a Polish space and $I$ a $\sigma$-ideal on the Borel subsets of $X$. A Borel set in $X$ is \textbf{$I$-positive} if it is not in $I$, and is \textbf{$I$-large} if its complement is in $I$. Let $\P_I$ be the poset of $I$-positive Borel subsets of $X$, ordered by inclusion. See \cite{Zapletal-2008-forcing-idealized} for the theory of forcing with $\P_I$. There is a canonical $\P_I$-name $\dot{x}$ for a member of $X$ added by a generic filter (see \cite[2.1.2]{Zapletal-2008-forcing-idealized}). We will say that $x\in X$ is $\P_I$-generic (over $V$) if $x=\dot{x}[G]$ for some generic filter $G\subset\P_I$ (over $V$).
Of particular interest are the ideals $I$ so that $\P_I$ is a proper poset (see \cite[Section~2.2]{Zapletal-2008-forcing-idealized}).

For the applications in this paper, $I$ will always be the ideal of meager sets, in which case the forcing $\P_I$ is isomorphic to Cohen-real forcing (and in particular is proper).

The following is a generalization of \cite[Lemma 2.5]{shani2019-strong-ergodicity}.
\begin{lem}\label{lem;symm-model-ergod}
Suppose $E$ and $F$ are analytic equivalence relations on Polish spaces $X$ and $Y$ respectively and $x\mapsto A_x$ is an absolute classification of $E$.
Let $I$ be a $\sigma$-ideal as above so that $\P_I$ is proper. Fix $x\in X$ a $\P_I$-generic over $V$, and let $A=A_x$.
Suppose $(\P,\tau)\in V(A)$ is definable from $A$ (alone) over $V$, $\P$ is a poset and $\tau$ is a $\P$ name for an element in $X$, such that in $V(A)$, $\P\force A_\tau=A$, and in $V[x]$, $\P$ is a reasonable poset.
The following are equivalent.
\begin{enumerate}
    \item For every partial Borel homomorphism $f\colon E\to_B F$, if $x$ is in the domain of $f$, then $f$ maps an $I$-positive set, containing $x$, into a single $F$-class;
    \item if $\sigma\in V(A)$ is definable from $A$ over $V$ and $(\P,\sigma)$ is an $F$-pin, then there is $y_0\in Y\cap V$ so that $\P\force \sigma\mathrel{F} \check{y}_0$.
\end{enumerate}
\end{lem}
\begin{remark}\label{remark: symmetric-model-ergod}
\begin{enumerate}
    \item In our examples $\P$ will be equivalent to Cohen forcing in $V[x]$ (but not in $V(A)$). In fact, $\P$ will be the ``quotient forcing'', so that $V[x]$ is a $\P$-generic extension of $V(A)$.
    \item In clause (1) of the lemma, if $x$ is in the domain of $f$ then the domain of $f$ is an $I$-positive set.
    \item Assume further that $E$ is $I$-ergodic, that is, any $E$-invariant Borel subset of $X$ is either in $I$ or its complement is in $I$. Then the conclusion in clause (1) is equivalent to: ``$f$ maps a $I$-large set to a single $F$-class''.
    \item When $I$ is the ideal of meager sets, then $I$-ergodicity is ``generic-ergodicity''. In our applications, the equivalence relation $E$ will be generically ergodic,  so (1) is equivalent to $E$ being generically $F$-ergodic.
\end{enumerate}
\end{remark}
\begin{proof}[Proof of Lemma~\ref{lem;symm-model-ergod}]
Assume (2), and let $f$ be a partial Borel homomorphism as in (1). Let $\sigma$ be the $\P$ name for $f(\tau)$. Then $\sigma$ is definable from $A$ over $V$ (since $\P$ and $\tau$ are so definable, and $f$ is coded in $V$). Also $(\P,\sigma)$ is an $F$-pin. By (2), there is some $y_0\in Y$ in the ground model $V$ so that $\P\force \sigma\mathrel{F}\check{y}_0$.
Fix a condition $p\in\P_I$ forcing this. Let $M$ be a sufficiently large transitive countable model. 
Let $C$ be the set of $z\in X$ which are $\P_I$-generic over $M$ and extend $p$. Then $C$ is $I$-positive by \cite[Proposition 2.2.2]{Zapletal-2008-forcing-idealized}. To conclude (1), we show that for any $x\in C$, $f(x)\mathrel{F}y_0$.
Fix $x\in C$, and let $H$ be some $\P$-generic over $V[x]$.
Since $A_x=A=A_{\tau[H]}$,  $x\mathrel{E}\tau[H]$. Also $\sigma[H]\mathrel{F}y_0$ and $\sigma[H]=f(\tau[H])$. Since $f\colon E\to_B F$ is a homomorphism, $f(x)\mathrel{F}y_0$.

Assume now (1), and fix a $\sigma$ as in (2). 
Working in $V[x]$, $(\P,\sigma)$ is an $F$-pin as well. Since $\P$ is reasonable, the pin $(\P,\sigma)$ is trivial, by \cite[Theorem~2.6.2]{Larson_Zapletal_2020}, so there is some $y\in V[x]$ such that $V[x]\models\P\force \sigma\mathrel{F}\check{y}$.
Let $\mu$ be a $\P_I$-name for $y$, and fix a condition $p$ in $\P_I$ forcing the above.
Fix a sufficiently large transitive countable model $M$, and let $C$ be the set of $z\in X\cap p$ which are $P_I$-generic over $M$. $C$ is $I$-positive by \cite[Proposition 2.2.2]{Zapletal-2008-forcing-idealized}. 
For $z\in C$ define $f(z)=\mu[z]$, the interpretation of $\mu$ according to the generic $z$ in $M[z]$.
Then $f$ is a partial Borel function, and $x$ is in the domain of $f$. 

We claim that $f$ is a partial homomorphism from $E$ to $F$. To see this, assume $z_1,z_2$ are in $C$ and are $E$-related. Then the invariants $A_{z_1}$ and $A_{z_2}$ are equal, which we denote by $\bar{A}$. Let $\bar{\P},\bar{\sigma}$ be as defined from $\bar{A}$ over $M$ in $M(\bar{A})$ (according to some fixed definition of $\P,\sigma$ from $A$ over $V$). Fix a generic filter $H\subset\P$ over $M(\bar{A})$.
Now $f(z_1)=\mu[z_1]\mathrel{F} \sigma[H]\mathrel{F}\mu[z_2]=f(z_2)$, as required.

Finally, by (1), there is an $I$-positive set $q\subset C$ with $x\in q$ and $y_0\in Y$ such that $f(z)\mathrel{F}y_0$ for all $z\in q$. 
Since $x$ is in $q$, then $y=\mu[x]\mathrel{F}y_0$, therefore $\P\force \sigma\mathrel{F} y_0$. 
\end{proof}


\section{A symmetric model for the $\Gamma$-jump}\label{section: a symmetric model}

Fix a generically ergodic countable Borel equivalence relation $E$ on a Polish space $X$, and $\Gamma$ a countable infinite group.
Towards Theorem~\ref{thm: main technical gen-erg thm}, we describe in this section the main model $V(A)$ which will be used to study the $\Gamma$-jump $E^{[\Gamma]}$.

Given $x\in X^{\Gamma}$, for each $\gamma\in\Gamma$ let $A^x_\gamma=[x(\gamma)]_E$, the $E$-class of $x(\gamma)$, and $\vec{A}^x=\seqq{A^x_\gamma}{\gamma\in\Gamma}$.
For a $\Gamma$-indexed sequence $\vec{A}=\seqq{A_\alpha}{\alpha\in\Gamma}$, define $\gamma\cdot \vec{A}=\seqq{A_{\gamma^{-1}\alpha}}{\alpha\in\Gamma}$.
Define
\begin{equation*}
A^x=\Gamma\cdot\vec{A}^x=\set{\gamma\cdot\seqq{A_{\alpha}}{\alpha\in\Gamma}}{\gamma\in \Gamma}. 
\end{equation*}
Recall from Section~\ref{section: prelims - classifying invariants} that the map $x\mapsto A^x$ is an absolute complete classification of $E^{[\Gamma]}$, while $x\mapsto \vec{A}^x$ is an absolute complete classification of $E^\Gamma$.

Let $x\in X^\Gamma$ be a Cohen-generic real over $V$, with respect to the product topology. Let $\vec{A}=\vec{A}^x$ and $A=A^x=\Gamma\cdot\vec{A}$, its $E^\Gamma$ and $E^{[\Gamma]}$ classifying invariants, respectively.
Note that $V(A)=V(\vec{A})$. This is because $A\in V(\vec{A})$ and $\vec{A}\in A$, and so $\vec{A}\in V(A)$.
However, while $A$ is definable from $\vec{A}$, $\vec{A}$ is not definable from $A$. We will often use $V(\vec{A})$ when dealing with definability from $\vec{A}$, and use $V(A)$ when dealing with definability from $A$.

This model was used in \cite{shani2019-strong-ergodicity} to study the $\Gamma$-jump $E^{[\Gamma]}$ and the product equivalence relation $E^\omega$, by studying objects in this model, definable from $A$ or $\vec{A}$ respectively. 
This same approach is taken here, combined with Lemma~\ref{lem;symm-model-ergod}, and a finer analysis of definable sets in $V(\vec{A})$.
The following lemma is a refinement of \cite[Proposition 3.5]{shani2019-strong-ergodicity}. 
\begin{lemma}[Interpolation in $V(\vec{A})$]\label{lemma: interpolation}
Let $F_0,F_1$ be finite subsets of $\Gamma$.
Let $\phi_i$ be formulas and $v_i$ parameters in $V$, for $i=0,1$.
Assume that in $V(\vec{A})$, for any $z$,
\begin{equation*}
    \phi_0(z,\vec{A},\seqq{x(\gamma)}{\gamma\in F_0},v_0)\iff \phi_1(z,\vec{A},\seqq{x(\gamma)}{\gamma\in F_1},v_1).
\end{equation*}
Then there is a formula $\psi$ and a parameter $w\in V$ such that in $V(\vec{A})$
\begin{equation*}
    \phi_0(z,\vec{A},\seqq{x(\gamma)}{\gamma\in F_0},v_0)\iff\psi(z,\vec{A},\seqq{x(\gamma)}{\gamma\in F_0\cap F_1},w).
\end{equation*}
\end{lemma}
\begin{proof}
Fix a condition $p$ forcing that $\phi_i^{V(\dot{\vec{A}})}(z,\dot{\vec{A}},\seqq{\dot{x}(\gamma)}{\gamma\in F_i},\check{v}_i)$ are equivalent for $i=0,1$.
Let $w\in V$ code $v_0$ and the condition $p$, and define $\psi(z,\vec{A},\seqq{x(\gamma)}{\gamma\in F_0\cap F_1},w)$ as the statement:
\begin{center}
``there are $y(\gamma)\in A_\gamma$, for $\gamma\in F_0$, such that $y(\gamma)$ agrees with $p(\gamma)$ for $\gamma\in F_0$, $y(\gamma)=x(\gamma)$ for $\gamma\in F_0\cap F_1$, and $\phi_0(z,\vec{A},\seqq{y(\gamma)}{\gamma\in F_0},v_0)$ holds''.    
\end{center}
We show that $\psi$ satisfies the conclusion of the lemma.
Note that, for any $z$, in $V(\vec{A})$
\begin{equation*}
    \phi_0(z,\vec{A},\seqq{x(\gamma)}{\gamma\in F_0},v_0)\implies \psi(z,\vec{A},\seqq{x(\gamma)}{\gamma\in F_0\cap F_1},w),
\end{equation*}
as witnessed by $y(\gamma)=x(\gamma)$ for $\gamma\in F_0$. We now prove the converse.
Assume that $\psi(z,\vec{A},\seqq{x(\gamma)}{\gamma\in F_1\cap F_2},w)$ holds in $V(\vec{A})$.
Fix $y(\gamma)$, $\gamma\in F_0$ such that $y(\gamma)$ agrees with $p(\gamma)$, $y(\gamma)=x(\gamma)$ for $\gamma\in F_0\cap F_1$, and $\phi_0^{V(\vec{A})}(z,\vec{A},\seqq{y(\gamma)}{\gamma\in F_0},v_0)$.

For $\gamma\in \Gamma\setminus F_0$, let $y(\gamma)=x(\gamma)$. Note that $\seqq{y(\gamma)}{\gamma\in\Gamma}\in X^\Gamma$ is Cohen-generic over $V$ and extends the condition $p$.
Note that $\dot{\vec{A}}$ is the same set $\vec{A}$ when interpreted by the generic $\seqq{x(\gamma)}{\gamma\in\Gamma}$ or the generic $\seqq{y(\gamma)}{\gamma\in\Gamma}$. 
Furthermore, $\dot{x}(\gamma)$ interpreted by $\seqq{y(\gamma)}{\gamma\in\Gamma}$ is $y(\gamma)$. 

Working in $V[\seqq{y(\gamma)}{\gamma\in \Gamma}]$, we conclude that $\phi_0^{V(\vec{A})}(z,\vec{A},\seqq{y(\gamma)}{\gamma\in F_0},v_0)$ holds, and that
\begin{equation*}
    \phi_0^{V(\vec{A})}(z,\vec{A},\seqq{y(\gamma)}{\gamma\in F_0},v_0)\iff \phi_1^{V(\vec{A})}(z,\vec{A},\seqq{y(\gamma)}{\gamma\in F_1},v_1) 
\end{equation*}
As $x(\gamma)=y(\gamma)$ for $\gamma\in F_1$, the right hand side is $\phi_1^{V(\vec{A})}(z,\vec{A},\seqq{x(\gamma)}{\gamma\in F_1},v_1)$, which is equivalent to $\phi_0^{V(\vec{A})}(z,\vec{A},\seqq{x(\gamma)}{\gamma\in F_0},v_0)$, as required.
\end{proof}
\begin{remark}
Each $A_\gamma$ is the $E$-class of $x(\gamma)$. Since $E$ is a countable Borel equivalence relation (whose code is in $V$), then for any two elements $z,u\in A_\gamma$ are definable from one another (using a parameter from $V$).
So in the lemma above, if we replace $\seqq{x(\gamma)}{\gamma\in F_i}$ with any sequence $\bar{x}_i\in\prod_{\gamma\in\Gamma_i}A_\gamma$, we reach the same conclusion.
\end{remark}

The following lemma, which is Lemma~5.1 in \cite{shani2019-strong-ergodicity}, captures the additional symmetries of $E^{[\Gamma]}$. We sketch a proof for convenience.
\begin{lemma}[Indiscernibility in $V(A)$]\label{lemma: indiscernibility V(A)}
In $V(A)$, the members of $A$ are indiscernible over $A$ and $V$.
\end{lemma}
\begin{proof}[Proof sketch]
Suppose $V(A)\models\phi(A,A_\gamma,v)$ for a formula $\phi$, a parameter $v\in V$ and some $\gamma\in \Gamma$. Fix $\delta\in\Gamma$. We need to show that $V(A)\models\phi(A,A_\delta,v)$.

Let $\Q$ be Cohen forcing for $X^\Gamma$ (which added the generic $x$). Fix $p\in\Q$ compatible with $x$, forcing that $\phi(\dot{A},\dot{A}_\gamma,\check{v})$.
Note that $\delta\gamma^{-1}$ acts as an automorphism of $\Q$ (by permuting the indices), fixing the name $\dot{A}$ (but not fixing $\dot{\vec{A}}$), sending $p$ to some $q$ which forces $\phi^{V(\dot{A})}(\dot{A},\dot{A}_{\delta},v)$.

The issue is that $q$ may not be compatible with $x$. Nevertheless, since $E$ is generically ergodic, we can change $x$ to some $x'$ so that $x'\mathrel{E^\Gamma}x$ (in fact we only change finitely many coordinates of $x$) and so that $x'$ extends $q$.

Note that $\dot{A}$ and $\dot{A}_\zeta$, for any $\zeta\in\Gamma$, are interpreted the same by the generics $x$ and $x'$. In particular, working in $V[x']$, since $q$ is in the generic, we conclude that $\phi^{V(A)}(A,A_\delta,v)$ holds, as required.
\end{proof}


\begin{defn}\label{defn: P,tau in V(A)}
In $V(A)$, let $\P$ be the poset of all conditions $p$ with $\dom p\subset \Gamma$ finite, $1_\Gamma\in\dom p$, and there is some $\vec{B}\in A$ such that $p\in\prod_{\gamma\in\dom p}\vec{B}_\gamma$. 

In other words, a condition $p$ chooses an origin $\gamma_0$ and then approximates a choice function through $\gamma_0\cdot \vec{A}$.
For $p\in\P$, let $\vec{A}(p)$ be the unique member of $A$ with $p(1_\Gamma)\in \vec{A}(p)_{1_\Gamma}$.
For two conditions $p,q\in\P$, say that $p$ extends $q$ if $p$ extends $q$ as a function. Note that it implies that  $\vec{A}(p)=\vec{A}(q)$.

Given a generic filter $G\subset\P$ over $V(A)$, $\bigcup G$ defines a choice function in a unique sequence in $A$.
Let $\tau$ be the name for this choice function $\bigcup G$. 
Then $\P\force A_\tau = A$.
In particular, $(\P,\tau)$ is an $E^{[\Gamma]}$-pin in $V(A)$.
\end{defn}

Note that in the Cohen-real extension $V[x]$ the poset $\P$ is countable (in particular reasonable). Also, $\P$ and $\tau$ are definable from $A$ over $V$. We are now in position to use Lemma~\ref{lem;symm-model-ergod}.


\section{Proof of Theorem~\ref{thm: main technical gen-erg thm}}

\begin{remark}
For an equivalence relation $E'$,
$E'$ is generically $F$-ergodic if and only if $E'$ is generically $F^{\Delta}$-ergodic. This is because a homomorphism to the product relation $F^\Delta$ can be identified with a $\Delta$-sequence of homomorphisms to $F$.
Homomorphisms to $F^{[\Delta]}$ cannot, in general, be understood so easily.
The main point in the proof of the theorem will be to show that, in these circumstances, a homomorphism to $F^{[\Delta]}$ can be ``converted'' to a homomorphism to $F^\Delta$, after restricting to a comeager set.
\end{remark}

Let $F$ be an analytic equivalence relation on a Polish space $Y$. Assume that for any group homomorphism $\phi$ from $\Gamma$ to a quotient of a subgroup of $\Delta$, its image is finite, and $E^{[\ker\phi]}$ is generically $F$-ergodic, for any generically ergodic countable Borel equivalence relation $E$.
Fix a generically ergodic countable Borel equivalence relation $E$ and consider the model $V(A)$ as in Section~\ref{section: a symmetric model}. We need to show that $E^{[\Gamma]}$ is generically $F^{[\Delta]}$-ergodic.
By Lemma~\ref{lem;symm-model-ergod}, it suffices to prove the following.
\begin{prop}
Suppose $\sigma$ is in $V(A)$, definable from $A$ over $V$, and $(\P,\sigma)$ is an $F^{[\Delta]}$-pin. Then there is $y_0\in Y^\Delta$ in $V$ such that $\P\force \sigma\mathrel{F^{[\Delta]}}\check{y}_0$.
\end{prop}
We prove this by showing that  $(\P,\sigma)$ is essentially an $F^\Delta$-pin, then using ergodicity with respect to $F^\Delta$, together with Lemma~\ref{lem;symm-model-ergod}.

By Lemma~\ref{lemma: pins from jump and product} there is a condition $p\in\P$ such that $(\P\restriction p,\sigma)$ is an $F^\Delta$-pin.
Without loss of generality, assume that $\vec{A}(p)=\vec{A}$.
Let $\bar{\Gamma}$ be the domain of $p$.
Consider now the statement: ``there is some $p\in\P$ with $\vec{A}(p)=\vec{A}$ and $\dom p=\bar{\Gamma}$ such that $(\P\restriction p,\sigma)$ is an $F^\Delta$-pin''.
By the Indiscernibility Lemma~\ref{lemma: indiscernibility V(A)}, this statement holds for any $\gamma\cdot\vec{A}$ in $A$.

There could be different conditions $p_0,p_1$ with the same domain $\bar{\Gamma}$ satisfying $\vec{A}(p_0)=\vec{A}(p_1)$ and that $(\P\restriction p_i,\sigma)$ is an $F^\Delta$-pin, yet the two pins $(\P\restriction p_0,\sigma)$ and $(\P\restriction p_1,\sigma)$ are not $F^\Delta$ equivalent to one another. 
We will now restrict to a definable subset of such conditions which give the same $F^\Delta$ pin, up to $F^\Delta$-equivalence.

Since $\Gamma$ is infinite, we may find $\gamma\in\Gamma$ with $\gamma\cdot\bar{\Gamma}$ and $\bar{\Gamma}$ disjoint. 
Fix $p,q\in\P$ with domains $\bar{\Gamma}$ such that $\vec{A}(q)=\gamma\cdot\vec{A}$, $\vec{A}(p)=\vec{A}$ and $(\P\restriction q,\sigma)$, $(\P\restriction p,\sigma)$ are $F^\Delta$ pins.
Since $(\P\restriction q,\sigma) \mathrel{F^{[\Delta]}} (\P\restriction p,\sigma)$,
there is $\delta\in\Delta$ such that
\begin{equation*}
    (\P\restriction q,\sigma) \mathrel{F^\Delta} (\P\restriction p,\delta\cdot \sigma).
\end{equation*}
Suppose now $p'\in\P\restriction\bar{\Gamma}$ satisfies $\vec{A}(p')=\vec{A}$ and $(\P\restriction p',\sigma)$ is an $F^\Delta$-pin. Then the following are equivalent:
\begin{enumerate}
    \item $(\P\restriction p',\sigma) \mathrel{F^\Delta} (\P\restriction p,\sigma)$;
    \item $(\P\restriction q,\sigma) \mathrel{F^\Delta} (\P\restriction p',\delta\cdot \sigma)$.
\end{enumerate}
The set of $p'$ for which (1) holds, is definable on the one hand using the parameter $p$ from $\vec{A}_{\bar{\Gamma}}$, and on the other hand using the parameter $q$ from $(\gamma\cdot\vec{A})_{\bar{\Gamma}}$, via clause (2). 
By Lemma~\ref{lemma: interpolation}, there is a formula $\psi$ and a parameter $w\in V$ such that (1) holds if and only if $\psi^{V(A)}(p',\vec{A},w)$ holds.

Working in $V(A)$, we now have a formula $\psi(p',\vec{A},w)$ so that for any $p_1,p_2$, if $\psi(p_i,\Vec{A},w)$ holds for $i=1,2$, then $p_i\in \P\restriction\bar{\Gamma}$, $\vec{A}(p_i)=\vec{A}$, and $(\P\restriction p_1,\sigma) \mathrel{F^\Delta} (\P\restriction p_2,\sigma)$.
It follows by indiscernibility that for any $\gamma\in\Gamma$, for any $p_1,p_2$, if $\psi(p_i,\gamma\cdot\vec{A},w)$ hold, then $p_i\in \P\restriction\bar{\Gamma}$, $\vec{A}(p_i)=\gamma\cdot\vec{A}$, and $(\P\restriction p_1,\sigma) \mathrel{F^\Delta} (\P\restriction p_2,\sigma)$.

Let $\P^\ast$ be all $p\in\P$ such that $\psi(p,\gamma\cdot\vec{A},w)$ holds for some $\gamma$.
Note that $\P^\ast$ is definable from $A$ and parameters in $V$ alone.
$\P^\ast$ is pre-dense in $\P$. In particular, we may identify $\sigma$ as a $\P^\ast$-name, and for $p\in\P^\ast$ the pins $(\P\restriction p,\sigma)$ and $(\P^\ast\restriction p,\sigma)$ are $F^\Delta$-equivalent.

For $p_1,p_2\in\P^\ast$, say that $\delta$ \textbf{connects} $\mathbf{p_2}$ to $\mathbf{p_1}$ if
\begin{equation*}
    (\P\restriction p_1,\sigma) \mathrel{F^\Delta} (\P\restriction p_2,\delta\cdot \sigma).
\end{equation*}
Fix $\vec{B}\in A$ (a shift of $\vec{A}$) and conditions $p,q$ with $\psi(p,\vec{B},w)$ and $\psi(q,\gamma\cdot\vec{B},w)$. 
If $\delta$ connects $p$ to $q$, then for any $p',q'$ satisfying $\psi(p',\vec{B},w)$ and $\psi(q',\gamma\cdot\vec{B},w)$, $\delta$ connects $p'$ to $q'$. In this case, say that $\delta$ \textbf{connects} $\vec{\mathbf{{B}}}$ to $\boldsymbol{\gamma}\cdot\vec{\mathbf{{B}}}$.
The statement ``$\delta$ connects $\vec{B}$ to $\gamma\cdot\vec{B}$'' is a statement about $\vec{B}$ involving only $A$ and parameters in $V$.
By indiscernibility, it holds for any $\vec{C}\in A$.
Finally, say that $\boldsymbol{\delta}$ \textbf{connects} $\boldsymbol{\gamma}$ if $\delta$ connects $\vec{A}$ to $\gamma\cdot\vec{A}$. 
By the discussion above, $\delta$ connects $\gamma$ if and only if $\delta$ connects $\vec{B}$ to $\gamma\cdot\vec{B}$, for any $\vec{B}\in A$.

\begin{prop}
\begin{enumerate}
    \item If $\delta$ connects $\gamma$ then $\delta^{-1}$ connects $\gamma^{-1}$;
    \item If $\delta_i$ connects $\gamma_i$ for $i=1,2$, then $\delta_1\delta_2$ connects $\gamma_1\gamma_2$.
\end{enumerate}
\end{prop}
\begin{proof}
(1) By assumption, $\delta$ connects $\vec{A}$ to $\gamma\cdot\vec{A}$. Therefore $\delta^{-1}$ connects $\gamma\cdot\vec{A}$ to $\vec{A}=\gamma^{-1}\cdot(\gamma\cdot\vec{A})$, and so $\delta^{-1}$ connects $\gamma^{-1}$.

(2) Since $\delta_i$ connects $\gamma_i$, $\delta_2$ connects $\vec{A}$ to $\gamma_2\cdot\vec{A}$ and $\delta_1$ connects $\gamma_2\cdot\vec{A}$ to $\gamma_1\cdot(\gamma_2\cdot\vec{A})$. It follows that $\delta_1\delta_2$ connects $\vec{A}$ to $\gamma_1\gamma_2\cdot\vec{A}$.
\end{proof}

Let $\tilde{\Delta}$ be all $\delta\in\Delta$ for which there is some $\gamma\in\Gamma$ such that $\delta$ connects $\gamma$.
For $\gamma\in\Gamma$, let $H_\gamma$ be all $\delta\in\tilde{\Delta}$ such that $\delta$ connects $\gamma$. Let $H=H_{1_\Gamma}$.
The previous proposition implies the following.
\begin{prop}
\begin{itemize}
    \item $\tilde{\Delta}$ is a subgroup of $\Delta$;
    \item $H$ is a normal subgroup of $\tilde{\Delta}$;
    \item Each $H_\gamma$ is a coset of $H$;
    \item The map $\gamma\mapsto H
    _\gamma$ is a group homomorphism from $\Gamma$ to $\tilde{\Delta}/H$.
\end{itemize}
\end{prop}
By assumption, the group homomorphism $\gamma\mapsto H_\gamma$ has finite image.

Assume first that this group homomorphism is \emph{in fact trivial}, that is, $H_\gamma=H$ for all $\gamma\in\Gamma$. (This is necessarily the case, for example, if $\Gamma=\mathbb{Z}_p^{<\omega}$ and $\Delta=\mathbb{Z}_q^{<\omega}$ for distinct primes $p$ and $q$.) Then $1_\Delta$ connects $\gamma$, for any $\gamma\in\Gamma$. It follows that for any $p,q\in\P^\ast$, $(\P\restriction q,\sigma) \mathrel{F^\Delta} (\P\restriction p,\sigma)$. Since $\P^\ast$ is predense in $\P$, we conclude that $(\P,\sigma)$ is in fact an $F^\Delta$-pin.
By assumption, $E^{[\Gamma]}$ is generically $F$-ergodic, and therefore generically $F^\Delta$-ergodic as well. 
Applying (1)$\implies$(2) of Lemma~\ref{lem;symm-model-ergod} for $E^{[\Gamma]}$ and $F^\Delta$, we find some $y_0\in Y^\Delta$ in $V$ such that $\P\force \sigma\mathrel{F^\Delta}\check{y}_0$. In particular, $\P\force \sigma\mathrel{F}^{[\Delta]}\check{y}_0$.
Applying (2)$\implies$(1) of Lemma~\ref{lem;symm-model-ergod} for $E^{[\Gamma]}$ and $F^{[\Delta]}$, we conclude that $E$ is generically $F^{[\Delta]}$-ergodic.

For the general case, when the group homomorphism is not trivial, we employ the following lemma.

\subsection{A lemma}
Let $E$ be an equivalence relation on $X$, $\Gamma$ a countable group.
Suppose $K\vartriangleleft\Gamma$ is a normal subgroup with finite index. Let $R\subset\Gamma$ be a (finite) set of representatives of the cosets.
Define a map $f\colon (X^R)^{K}\to X^{\Gamma}$ as follows. For $x\in (X^R)^{K}$, $\gamma\in K$ and $a\in R$,
\begin{equation*}
    f(x)(\gamma\cdot a)=x(\gamma)(a).
\end{equation*}
Note that $f$ is a homeomorphism from $(X^R)^{K}$ to $X^\Gamma$.
\begin{lemma}\label{lemma: technical jumps normal subgroup}
Suppose $x,y\in (X^R)^{K}$, $\zeta\in K$, and $\zeta\cdot x\mathrel{(E^R)^{K}}y$. Then $\zeta\cdot f(x) \mathrel{E}^\Gamma f(y)$. 
In particular, $f$ is a homomorphism from $(E^R)^{[K]}$ to $E^{[\Gamma]}$.
\end{lemma}
Here, $E^R$ is the point-wise product equivalence relation.
\begin{proof}
Suppose $x,y\in (X^R)^{K}$, $\zeta\in K$ such that $\zeta\cdot x\mathrel{(E^R)^{K}}y$. That is,
\begin{equation*}
    x(\zeta^{-1}\gamma)(a)\mathrel{E}y(\gamma)(a)
\end{equation*}
for any $\gamma\in K$ and $a\in R$. 
Given $\gamma\in \Gamma$, fix $a\in R$ and $\gamma_0\in K$ such that $\gamma=\gamma_0\cdot a$. Then
\begin{equation*}
    f(y)(\gamma)=f(y)(\gamma_0\cdot a)=y(\gamma_0)(a)\mathrel{E}x(\zeta^{-1}\gamma_0)(a)=f(x)(\zeta^{-1}\gamma_0\cdot a)=f(x)(\zeta^{-1}\gamma).
\end{equation*}
That is, $\zeta\cdot f(x)\mathrel{E^\Gamma}f(y)$. 
\end{proof}

\subsection{Conclusion of the proof}

We now return to the proof. Let $K$ be the Kernel of the homomorphism $\gamma\mapsto H_\gamma$. 
Define $\Q=\set{p\in \P}{\vec{A}(p)=\gamma\cdot\vec{A}\textrm{ for some }\gamma\in K}$, and $\Q^\ast=\Q\cap \P^\ast$.
Then $\Q^\ast$ is pre-dense in $\Q$, and any $\Q$-generic can be identified as a $\P$-generic.
In fact, for any condition $q\in\Q$, $\P\restriction q = \Q\restriction q$.
We may therefore consider $\tau$ and $\sigma$ as a $\Q^\ast$-name as well.
By the construction above, for any two conditions $p,q\in\Q^\ast$, $(\Q^\ast\restriction p,\sigma)$ and $(\Q^\ast\restriction q, \sigma)$ are $F^\Delta$-related. In other words, $(\Q^\ast,\sigma)$ is an $F^\Delta$-pin.

By assumption, $K$ has finite index.
Fix a finite set $R\subset \Gamma$ of representatives of the cosets of $K$.
Let $f\colon (X^R)^K\to X^\Gamma$ be the homeomorphism as above.
The product equivalence relation $E^R$ on $X^R$ is a generically ergodic countable Borel equivalence relation.
By assumption, $(E^R)^{[K]}$ is generically $F$-ergodic.

Recall that $A=A_x$ for a Cohen generic $x\in X$. Let $x'=f^{-1}(x)$, a Cohen generic element in $(X^R)^K$.
Let $A'$ be the $(E^R)^{[K]}$ -classifying invariant of $x'$.
Recall that it is defined as follows. Let $A'_\gamma=[x'(\gamma)]_{E^R}$ for $\gamma\in K$. We may identify $A'_\gamma$ with the finite sequence $\seqq{A_{\gamma\cdot a}}{a\in R}$. Let $\vec{A}'=\seqq{A'_\gamma}{\gamma\in K}$. Then $A'=K\cdot \vec{A}'=\set{\gamma\cdot\vec{A}'}{\gamma\in K}$.

Note that $V(A')=V(A)=V(\vec{A})$, as $A'$ can be defined from $\vec{A}$. Furthermore, $A$ is definable from $A'$ over $V$, but not the other way around.
More importantly, the poset $\Q^\ast$ is definable from $A'$ over $V$. In fact, let $\P'$ be defined for $A'$ as $\P$ was defined for $A$. Then $\P'$ and $\Q$ are forcing isomorphic and bi-definable from one another.

In conclusion, $(\Q^\ast,\tau)$ is definable from $A'$, and satisfies $\Q^\ast\force A_\tau=A'$. Furthermore, $(\Q^\ast,\sigma)$ is an $F^\Delta$-pin where $\sigma$ is definable from $A'$ over $V$.
By (1)$\implies$(2) of Lemma~\ref{lem;symm-model-ergod} for $E^{[K]}$ and $F^\Delta$, there is $y_0\in Y^\Delta$ such that $\Q^\ast\force \sigma\mathrel{F^\Delta}y_0$. 
Since $(\Q^\ast,\sigma)$ and $(\P,\sigma)$ are $F^{[\Delta]}$-equivalent, we conclude that $\P\force \sigma\mathrel{F^{[\Delta]}}\check{y}_0$, as desired.

\bibliographystyle{alpha}
\bibliography{bibliography}

\end{document}